\documentclass[11pt]{article}

\usepackage[utf8]{inputenc}
\usepackage[T1]{fontenc}
\usepackage[english]{babel}
\usepackage[left=2cm,right=2cm,top=2cm,bottom=2cm]{geometry}
\usepackage[sc]{mathpazo}
\usepackage{amsthm,mathtools,amsfonts,amssymb} 
\usepackage{mathrsfs}
\usepackage{stmaryrd} 
\usepackage{faktor} 
\usepackage{fourier-orns} 

\usepackage{xcolor}
\definecolor{darkpurple}{HTML}{5a4a9c}
\definecolor{darkblue}{HTML}{186294}
\usepackage[colorlinks=true,breaklinks=true,citecolor=darkblue,linkcolor=darkpurple]{hyperref} 
\usepackage{cleveref} 
\usepackage{url} 
\usepackage{enumitem} 
\setenumerate{topsep=0pt,itemsep=2pt,parsep=0pt,partopsep=0pt}
\setitemize{topsep=0pt,itemsep=2pt,parsep=0pt,partopsep=0pt,label=\protect\textbullet}
\setdescription{topsep=0pt,itemsep=7pt,parsep=0pt,partopsep=0pt}

\usepackage{pgf,tikz}
\usepackage{pgfplots}
\pgfplotsset{compat=1.18}

\newtheoremstyle{courant} 
{3mm} 
{2mm} 
{\normalfont} 
{0pt} 
{\bfseries} 
{.} 
{ } 
{} 

\theoremstyle{courant}

\newtheorem{theorem}{Theorem}[section]
\newtheorem{lemma}[theorem]{Lemma}
\newtheorem{proposition}[theorem]{Proposition}
\newtheorem{corollary}[theorem]{Corollary}
\newtheorem{definition}[theorem]{Definition}
\newtheorem{example}[theorem]{Example}

\newtheorem{remark}[theorem]{Remark}

\newtheorem*{question*}{Question}

\makeatletter
\renewenvironment{proof}[1][\proofname]{\vspace{-3mm}\par
	\pushQED{\qed}%
	\normalfont \topsep6\p@\@plus6\p@\relax
	\trivlist
	\item\relax
	{\normalfont\textit
		{#1}\@addpunct{.}\ }}{\popQED\endtrivlist\@endpefalse
}
\makeatother

\renewcommand{\d}{\mathrm{d}}
\renewcommand{\Im}{\mathrm{Im}\,}
\renewcommand{\Re}{\mathrm{Re}\,}
\renewcommand{\O}{\mathrm{O}}
\renewcommand{\L}{\mathrm{L}}
\newcommand{\C}{\mathbf{C}}
\newcommand{\R}{\mathbf{R}}

\newcommand{\N}{\mathbf{N}}
\newcommand{\Z}{\mathbf{Z}}
\renewcommand{\S}{\mathbf{S}}
\renewcommand{\H}{\mathbf{H}}

\newcommand{\F}{\mathbf{F}}
\newcommand{\ie}{\textit{i.e. }}
\newcommand{\action}{\curvearrowright}
\DeclareMathOperator{\Isom}{Isom}
\DeclareMathOperator{\Hom}{Hom}
\DeclareMathOperator{\Conv}{Conv}

\DeclareMathOperator{\Min}{Min}
\DeclareMathOperator{\Id}{Id}
\DeclareMathOperator{\Fix}{Fix}

\DeclareMathOperator{\PO}{PO}

\DeclareMathOperator{\GL}{GL}
\DeclareMathOperator{\SL}{SL}

\DeclareMathOperator{\PSL}{PSL}
\DeclareMathOperator{\Aut}{Aut}
\DeclareMathOperator{\supp}{supp}
\DeclareMathOperator{\CAT}{CAT}

\DeclareMathOperator{\DF}{DF}
\DeclareMathOperator{\Lie}{Lie}
\DeclareMathOperator{\Jac}{Jac}
\DeclareMathOperator{\QI}{QI}
\DeclareMathOperator{\CC}{CC}
\DeclareMathOperator{\Cay}{Cay}
\DeclareMathOperator{\id}{id}

\newcommand{\compactcdot}{\mkern-3mu\cdot\mkern-3mu}

\newcommand*{\transp}[2][-3mu]{\ensuremath{\mskip1mu\prescript{\smash{\mathrm t\mkern#1}}{}{\mathstrut#2}}}

\makeatletter
\DeclareRobustCommand*{\mfaktor}[3][]
{
   { \mathpalette{\mfaktor@impl@}{{#1}{#2}{#3}} }
}
\newcommand*{\mfaktor@impl@}[2]{\mfaktor@impl#1#2}
\newcommand*{\mfaktor@impl}[4]{
   \settoheight{\faktor@zaehlerhoehe}{\ensuremath{#1#2{#3}}}%
   \settoheight{\faktor@nennerhoehe}{\ensuremath{#1#2{#4}}}%
      \raisebox{-0.5\faktor@zaehlerhoehe}{\ensuremath{#1#2{#3}}}%
      \mkern-4mu\diagdown\mkern-5mu%
      \raisebox{0.5\faktor@nennerhoehe}{\ensuremath{#1#2{#4}}}%
}
\makeatother

\title{Convex-cocompact representations into the isometry group of the infinite-dimensional hyperbolic space}
\author{\sc David Xu}
\date{}

\begin{document}

\maketitle

\begin{abstract}
We prove that convex-cocompact representations of finitely generated groups in the group of isometries of the infinite-dimensional hyperbolic space form an open set in the space of representations, allowing us to deform these convex-cocompact representations. We then use bending to obtain convex-cocompact representations of surface groups that are not conjugate to any exotic representation of PSL(2,R) classified by Monod and Py.
\end{abstract}

\tableofcontents

\section{Introduction}
Convex-cocompact representations of fundamental groups of hyperbolic surfaces into rank-one Lie groups have been extensively studied as natural generalizations of Fuchsian representations. If $\Gamma$ is the fundamental group of a closed surface $S$ (compact connected orientable and without boundary components) of genus at least $2$, classifying the hyperbolic structures on $S$ is equivalent to classifying the discrete and faithful representations of $\Gamma$ into $\PSL(2,\R)$. These representations are defined up to the action by conjugation of $\PSL(2,\R)$. Denote by $\DF(\Gamma,\PSL(2,\R))$ the space of discrete and faithful representations of $\Gamma$ to $\PSL(2,\R)$, the quotient space
$$\mathcal{T}(S) = \DF(\Gamma,\PSL(2,\R))/\PSL(2,\R)$$
is called the \emph{Teichmüller space of $S$} and consists of representations of $\Gamma$ having cocompact image in $\PSL(2,\R)$. The space $\DF(\Gamma,\PSL(2,\R))$ inherits a topology as a subset of $\Hom(\Gamma,\PSL(2,\R))$ and $\mathcal{T}(S)$ is then equipped with the quotient topology. If $g \geqslant 2$ is the genus of $S$, then $\mathcal{T}(S)$ is homeomorphic to $\R^{6g-6}$. One notable feature of the Teichmüller space is that it is a whole connected component in the character variety $\Hom(\Gamma,\PSL(2,\R))/\PSL(2,\R)$.

Representations in $\mathcal{T}(S)$ can be deformed inside $\PSL(2,\C)$, these deformations are called \emph{quasi-Fuchsian} representations. The space of quasi-Fuchsian representations coincides with that of convex-cocompact represesentations of $\Gamma$ into $\PSL(2,\C)$, \ie representations whose image preserve some convex subset of $\H^{3}$ and act cocompactly on it. Moreover, this space is open in the character variety $\Hom(\Gamma,\PSL(2,\C))/\PSL(2,\C)$.

Note that $\PSL(2,\R)$ and $\PSL(2,\C)$ are isomorphic to the groups $\Isom^{+}(\H^{2})$ and $\Isom^{+}(\H^{3})$ of orientation-preserving isometries of the hyperbolic plane and hyperbolic $3$--space respectively. More generally, one can consider discrete representations of the group $\Gamma$ into the isometry groups of higher-dimensional hyperbolic spaces, $\Isom(\H^{n})$. The set of convex-cocompact representations of $\Gamma$ into $\Isom(\H^{n})$ is again open. This property is called the \emph{stability} of convex-cocompact representations, it is attributed to Marden for $n=3$ and to Thurston for the general case $n\in \N$.

In this paper, we are interested in describing the space of convex-cocompact representations of a surface group $\Gamma$ into the isometry group of the infinite-dimensional hyperbolic space. This space, denoted by $\H^{\infty}$, is an infinite-dimensional analog of the hyperbolic spaces $\H^{n}$ and its study was suggested by Gromov in \cite[Section 6]{Gromov}.

Although $\Isom(\H^{\infty})$ is not proper (closed balls are not compact), convex-cocompact subgroups in $\Isom(\H^{\infty})$ do exist. Take a totally geodesic embedding of $\H^{n}$ into $\H^{\infty}$, this comes with an embedding of their groups of isometries. Then any cocompact lattice $\Gamma$ in $\Isom(\H^{n})$ preserves this copy of $\H^{n}$ and acts cocompactly on it. One can actually find more ways to embed $\Isom(\H^{n})$ into $\Isom(\H^{\infty})$. In \cite{DP,MP}, Delzant and Py and later Monod and Py constructed families of irreducible representations of the isometry groups $\Isom(\H^{n})$ into $\Isom(\H^{\infty})$, \ie representations whose images act minimally on $\H^{\infty}$ (without any invariant totally geodesic proper subspace) and have no fixed point in the boundary $\partial \H^{\infty}$. These representations were called \emph{exotic} in \cite{MP}, we will also use this terminology in the remainder of this article. One can also construct similar representations for groups of automorphisms of a tree $\mathcal{T}$ (see \cite{BIM}). Restricting these representations to a discrete subgroup $\Gamma$ of $\Isom(\H^{n})$ or $\Aut(\mathcal{T})$, we obtain a \emph{strongly discrete} subgroup of $\Isom(\H^{\infty})$ acting irreducibly on $\H^{\infty}$. Strong discreteness will be defined along with other notions of discreteness in Section \ref{sec:preliminaries}. If $\Gamma$ is cocompact, its image in $\Isom(\H^{\infty})$ via the representations will be convex-cocompact.

In finite dimension, as mentioned above, the space of convex-cocompact representations of a finitely generated group $\Gamma$ into $\Isom(\H^{n})$ is open. The proof presented in \cite[Theorem 11.4]{Canary} relies on the fact that the convex-cocompact representations coincide with the representations $\rho$ such that the orbit map $\tau_{\rho} : \gamma \mapsto \rho(\gamma)(x_{0})$ is a quasi-isometric embedding in $\H^{n}$, where $x_{0} \in \H^{n}$ is any base point. We prove that this correspondence also holds for convex-cocompact representations into $\Isom(\H^{\infty})$.

\begin{theorem}\label{thm:QI<=>convex-cocompact}
Let $\Gamma$ be a finitely generated group and $\rho : \Gamma \to \Isom(\H^{\infty})$ be a representation. The following two assertions are equivalent.
\begin{enumerate}
    \item The orbit map $\tau_{\rho}:\gamma \mapsto \rho(\gamma)(x_{0})$ for any $x_{0} \in \H^{\infty}$ is a $\Gamma$-equivariant quasi-isometric embedding from $\Gamma$ into $\H^{\infty}$, where $\Gamma$ is endowed with the word distance associated to a finite set of generators. 
    \item There is a convex and locally compact $\Gamma$-invariant subset $\mathcal{C}\subset \H^{\infty}$ on which $\Gamma$ acts cocompactly and $\rho(\Gamma)$ is strongly discrete in $\Isom(\H^{\infty})$. 
\end{enumerate}
\end{theorem}

Let $\QI(\Gamma,\Isom(\H^{\infty}))$ be the set of representations of $\Gamma$ into $\Isom(\H^{\infty})$ which have quasi-isometric orbit maps and let $\CC(\Gamma,\Isom(\H^{\infty}))$ be the set of convex-cocompact representations in $\Hom(\Gamma,\Isom(\H^{\infty}))$. We can deduce that convex-cocompact representations of finitely generated groups into $\Isom(\H^{\infty})$ form an open subset of the space of all the representations.

\begin{corollary}\label{cor:openness thm}
If $\Gamma$ is finitely generated, then $\QI(\Gamma,\Isom(\H^{\infty})) = \CC(\Gamma,\Isom(\H^{\infty}))$ and these subsets are open in $\Hom(\Gamma,\Isom(\H^{\infty}))$.
\end{corollary}

It is already a little surprising that irreducible convex-cocompact representations could exist in infinite dimension since $\H^{\infty}$ is not proper. The constructions in \cite{BIM,DP} displayed families of convex-cobounded representations for $\Aut(\mathcal{T})$ and $\SL(2,\R)$, it was then proved in \cite{MP} that both of them actually come along with some locally compact subset of $\H^{\infty}$ on which the groups act minimally. The stability of convex-cocompact representations then allows the use of deformations to construct new convex-cocompact subgroups of $\Isom(\H^{\infty})$. 

Using bendings, which are techniques developed by Johnson and Millson \cite{JM} to deform representations of some lattices of $\PO(n,1)$ into $\PO(m,1)$ with $m > n$, we are able to prove the following.

\begin{theorem}\label{thm:infinite dimensional space of deformation}
    Let $\Gamma$ be the fundamental group of a closed hyperbolic surface. If $\rho : \Isom(\H^{2}) \to \Isom(\H^{\infty})$ is one of the exotic representations described by Monod and Py, then the space of deformations of $\rho|_{\Gamma}$, up to conjugation in $\Isom(\H^{\infty})$, contains a one-parameter family of representations which are not pairwise conjugate, nor conjugate to the restriction of any exotic representation of $\Isom(\H^{2})$.
\end{theorem}

Our theorem actually shows that for (torsion-free) cocompact lattices in $\Isom(\H^{2})$, there are much more representations than for the whole group $\Isom(\H^{2})$ itself, as classified by Monod and Py \cite{MP}. This can be thought of as a non-rigidity or flexibility result for representations of lattices of $\Isom(\H^{2})$ into $\Isom(\H^{\infty})$. Besides, the fact that the representations of the surface groups $\Gamma < \Isom(\H^{2})$ obtained by deformations do not arise as the restriction to $\Gamma$ of any exotic representation $\PSL(2,\R) \to \Isom(\H^{\infty})$ described by Monod and Py illustrates that the representations of surface groups in $\Isom(\H^{\infty})$ are much richer than that of $\PSL(2,\R)$.

To prove that some representations obtained by bending are not conjugate to each other, we first describe the centralizer of a loxodromic isometry in $\Isom(\H^{\infty})$. This centralizer is an infinite-dimensional Lie group (also called \emph{Banach-Lie groups}) for the uniform operator norm, \ie it is a group which is also a manifold modeled on an infinite-dimensional Banach space, where the manifold structure is compatible with the group operations. The Lie algebra of an infinite-dimensional Lie group is well-defined as well as the exponential map, however there are some notable differences with the classical finite-dimensional case: a closed subgroup of an infinite-dimensional Lie group is not always a Lie group; and an infinite-dimensional Lie algebra is not necessarily the Lie algebra of some infinite-dimensional Lie group.

Harris and Kaup \cite{harris-kaup} studied infinite-dimensional Lie groups that are algebraic, \ie defined by polynomials (possibly infinitely many, but with uniformly bounded degree). In this case, the Lie groups resemble those in finite dimension. The Lie group involved in our space of deformations is defined by only two polynomials. We are then able to compute its Lie algebra and show that it has infinite dimension. Some more materials on infinite-dimensional Lie groups can be found in \cite{delaharpe,kac_groups}.

This centralizer contains a one-parameter family of isometries having distinct translation length. Bending along this family then provides non-conjugate and continuous deformations of the initial representation.

\paragraph{Acknowledgement.}
We wish to thank Bruno Duchesne and Gye-Seon Lee for their constant support and help throughout the course of this work. We are very grateful to Gilles Courtois and Antonin Guilloux for their interest in this work. We are also indebted to Pierre Py who brought up the question of stability of convex-cocompact representations in the infinite-dimensional hyperbolic space, provided us with his idea of using bendings and who pointed out a significant mistake in the previous version of the paper.

The author is supported by the KIAS Individual Grant (MG100801) at Korea Institute for Advanced Study.

This work was also supported by a PHC Star grant and by the National Research Foundation of Korea (NRF) grant funded by the Korea government (Ministry of Science and ICT) (No. PHC 50166PH and RS-2023-00259480).

\section{Preliminaries}\label{sec:preliminaries}

We define the hyperbolic space of infinite dimension and describe some notions of discreteness for a group of isometries acting on a non-proper space.

\subsection{The infinite-dimensional hyperbolic space}

Let $\mathcal{H}$ be a real separable Hilbert space and denote by $(e_{i})_{i\in\N}$ a Hilbert basis of $\mathcal{H}$. Let $Q$ be a non-degenerate quadratic form on $\mathcal{H}$ of signature $(\infty,1)$ defined by
$$Q(x) = -x_{0}^{2} + \sum_{i=1}^{\infty} x_{i}^{2}$$
where the $x_{i}$'s are the coordinates of $x \in \mathcal{H}$ in the given basis. The (separable) infinite-dimensional real hyperbolic space is
$$\H^{\infty} = \{x \in \mathcal{H} \mid Q(x) = -1, x_{0} > 0\}.$$
This space is a non-proper complete totally geodesic metric space which is $\CAT(-1)$. The hyperbolic distance on $\H^{\infty}$ is given by the formula
$$\cosh(d(x,y)) = -B(x,y) = -\langle x,Jy\rangle,$$
where $B$ is the bilinear form obtained by polarization of $Q$, $\langle\cdot,\cdot\rangle$ denotes the scalar product on $\mathcal{H}$ and $J$ is the linear map on $\mathcal{H}$ such that $Je_{0} = -e_{0}$ and $Je_{i} = e_{i}$ for all $i\geqslant 1$. The group of isometries of $\H^{\infty}$ is $\Isom(\H^{\infty}) = \PO(\infty,1)$, the projectivization of the group of linear automorphisms of
$\mathcal{H}$ that preserve $Q$.

This definition is the analog of the hyperboloid model for finite-dimensional hyperbolic spaces. The other usual models for hyperbolic spaces can also be extended to infinite dimension. These models are described in \cite{DSU}, as well as many other properties of this space. In particular, the Klein model (or projective model) will be useful when dealing with convex subsets of the hyperbolic space in the proof of Theorem \ref{thm:QI<=>convex-cocompact}. We may identify $\H^{\infty}$ with the unit ball $\mathcal{B}$ of a real separable Hilbert space $\mathcal{H}$. The cone topology on $\overline{\H^{\infty}}$ corresponds to the topology of the Euclidean norm on the closed unit ball $\overline{\mathcal{B}}$.

\subsection{Discreteness of subgroups}

There are several equivalent ways to define discreteness for subgroups of isometries of finite-dimensional hyperbolic spaces. However, they do not agree in general for subgroups in $\Isom(\H^{\infty})$.

One natural group topology on $\Isom(\H^{\infty})$ is the \emph{compact-open topology}. This topology is actually equivalent to the \emph{pointwise convergence topology} as well as the quotient of the \emph{strong operator topology} on the group of operators $\O(\infty,1)$ (see \cite[Proposition 5.1.2]{DSU}). Topological properties of this group were studied by Duchesne in \cite{Duc23}. In particular, $\Isom(\H^{\infty})$ with the compact-open topology is a \emph{Polish} group, that is, a completely metrizable and separable group. Thus a subgroup $\Gamma < \Isom(\H^{\infty})$ can be called discrete if it is discrete as a topological subspace of $\Isom(\H^{\infty})$ with this compact-open topology. Keeping the notations in \cite{DSU}, we will denote this by \emph{$COT$-discrete} or \emph{$COTD$} (where $COT$ stands for compact-open topology).

For any metric space $X$, using the action of $\Isom(X)$ on $X$, one has other ways to define discreteness. We will again use the same terminology and notations as in \cite{DSU}. 

\begin{definition}[\cite{DSU}, Definition 5.2.1 and Remark 5.2.3]
Let $X$ be a metric space and $\Gamma < \Isom(X)$.
\begin{itemize}
    \item We say that $\Gamma$ is \emph{strongly discrete} ($SD$) if for every bounded set $B \subset X$, we have
    $$\#\{\gamma \in \Gamma \mid \gamma(B) \cap B \neq \emptyset\} < \infty.$$
    Equivalently, $\Gamma$ is strongly discrete if for every $x\in X$ and $R > 0$,
    $$\#\{\gamma \in \Gamma \mid d(x,\gamma(x))\leqslant R\} < \infty.$$
    \item We say that $\Gamma$ is \emph{moderately discrete} ($MD$) if for every $x \in X$, there exists an open set $U\subset X$ containing $x$ such that
    $$\#\{\gamma \in \Gamma \mid \gamma(U) \cap U \neq \emptyset\} < \infty.$$
    Equivalently, $\Gamma$ is moderately discrete if for every $x\in X$, there exists $R > 0$ such that
    $$\#\{\gamma \in \Gamma \mid d(x,\gamma(x))\leqslant R\} < \infty.$$
    \item We say that $\Gamma$ is \emph{weakly discrete} ($WD$) if for every $x \in X$, there exists $\epsilon > 0$ such that
    $$\Gamma\cdot x \cap B(x,\epsilon) = \{x\}.$$
\end{itemize}
\end{definition}

\begin{remark}
    Note that $(SD)$ is also known as \emph{metrically proper}, $(MD)$ as \emph{wandering} and that $(WD)$ means having discrete orbits.
\end{remark}

\begin{proposition}[\cite{DSU}, Propositions 5.2.4 and 5.2.7]
For $X = \H^{\infty}$, let $\Gamma < \Isom(\H^{\infty})$. We have the following implications
$$SD \Rightarrow MD \Rightarrow WD \Rightarrow COTD,$$
where none of the reverse implications holds.
\end{proposition}

More details can be found in \cite[Chapter 5]{DSU}.

\section{Stability of convex-cocompact representations into \texorpdfstring{$\Isom(\H^{\infty})$}{Isom(H^infty)}}

The space $\H^{\infty}$ is an infinite-dimensional manifold, modelled on a Hilbert space. Thus it is not locally compact and not proper (closed balls are not compact). It then becomes more difficult to find compact subsets inside $\H^{\infty}$ which do not lie inside a finite-dimensional hyperbolic space $\H^{n}$ embedded in $\H^{\infty}$. In this non-proper setting, it seems then more natural to look for bounded sets instead of compact ones. Therefore, we start by defining convex-cobounded actions. However, we will see that the finite-generation assumption of the group allows us to recover actual convex-cocompactness. We emphasize that cocompactness and coboundedness are different in general at the end of this section.

\subsection{Convex-cocompact representations and quasi-isometric embeddings}

\begin{definition}
Let $\Gamma$ be a group acting by isometries on a metric space $X$. We say that $\Gamma$ is \emph{cobounded} if there exists $\sigma > 0$ such that
$$\Gamma\cdot B(0,\sigma) = X.$$
Moreover, we say that $\Gamma$ is \emph{convex-cobounded} if there is a $\Gamma$-invariant convex subset $\mathcal{C}\subset X$ such that the restriction of $\Gamma$ to $\mathcal{C}$ is cobounded.
\end{definition}

Recall that a map $f : X \to Y$ between two metric spaces $(X,d_{X})$ and $(Y,d_{Y})$ is called a \emph{$(K,C)$--quasi-isometric embedding} if there exist $K\geqslant 1$ and $C \geqslant 0$ such that for all $x_{1}, x_{2} \in X$,
$$\dfrac{1}{K}d_{Y}(f(x_{1}),f(x_{2}))-C \leqslant d_{X}(x_{1},x_{2}) \leqslant Kd_{Y}(f(x_{1}),f(x_{2})) + C.$$
If, in addition, $f$ is coarsely surjective, \ie there exists a constant $L \geqslant 0$ such that every point of $Y$ lies in the $L$–neighbourhood of the image of $f$, then $f$ is called a \emph{$(K,C)$--quasi-isometry}.

\begin{proposition}[\cite{GH}, Proposition 7.14] \label{prop:extension_to_boundary}
Let $X,Y$ be two $\delta$-hyperbolic geodesic spaces. If $\varphi: X \to Y$ is a quasi-isometry, then $\varphi$ induces a boundary map $\partial\varphi: \partial X \to \partial Y$ which is a homeomorphism.
\end{proposition}

If $\Gamma$ is a group acting properly and cocompactly on some proper geodesic space $X$, then the \v{S}varc-Milnor lemma states that $\Gamma$ must be finitely generated and that for any choice of base point $x_{0}\in X$, the orbit map $\gamma \mapsto \gamma(x_{0})$ is a quasi-isometry between the Cayley graph of $\Gamma$ (for some finite generating set $S$) and $X$, see for example \cite[Proposition I.8.19]{BH}. Recall that two elements $\gamma_{1},\gamma_{2} \in \Gamma$ are connected by an edge in the Cayley graph $\Cay_{S}(\Gamma)$ whenever there is an element $s \in S\cup S^{-1}$ such that $\gamma_{1} = \gamma_{2}s$ and that the metric $d_{S}$ on $\Cay_{S}(\Gamma)$ is the metric that assigns length $1$ to every edge in $\Cay_{S}(\Gamma)$.

When the space is not proper, there is an analog of \v{S}varc-Milnor's lemma for a convex-cobounded action instead of cocompact.

\begin{theorem}[\cite{DSU}, Theorem $12.2.12$]\label{thm:svarc-milnor version DSU}
Let $X$ be a $\CAT(-1)$ space, let $x_{0} \in X$ be a base point. Suppose that $\Gamma < \Isom(X)$ is strongly discrete and convex-cobounded. Then $\Gamma$ is finitely generated and the orbit map $\gamma\mapsto\gamma(x_{0})$ is a quasi-isometric embedding.
\end{theorem}

\begin{remark}
When $X = \H^{\infty}$, a subgroup $\Gamma < \Isom(\H^{\infty})$ cannot be simultaneously cobounded (for its action by isometries on $\H^{\infty}$) and strongly discrete, see \cite[Prop. 12.2.2]{DSU}.
\end{remark}

\begin{proposition} \label{prop:QI=>strongly discrete+almost faithful}
Let $\Gamma$ be a finitely generated group and $S$ be a finite generating set. If $\rho : \Gamma \to \Isom(\H^{\infty})$ is a representation such that the orbit map is a quasi-isometric embedding, then $\rho(\Gamma)$ is strongly discrete. In particular, the representation $\rho$ is almost faithful (\ie the kernel of $\rho$ is finite).
\end{proposition}

\begin{proof}
The proof does not differ from the finite-dimensional case. Let $x_{0} \in \H^{\infty}$ and suppose that the orbit map $\tau_{\rho}: \Gamma \to \H^{\infty}, \gamma \mapsto \rho(\gamma)(x_{0})$ is a $(K,C)$--quasi-isometric embedding. Let $R > 0$. Since $\tau_{\rho}$ is $(K,C)$--quasi-isometric, we have
$$\dfrac{1}{K}d_{S}(id,\gamma)-C \leqslant d(\tau_{\rho}(id),\tau_{\rho}(\gamma)) = d(x_{0},\rho(\gamma)(x_{0}))$$
so $d(x_{0},\rho(\gamma)(x_{0})) \leqslant R$ implies that $d_{S}(id,\gamma) \leqslant K(R+C)$, thus
$$\#\{\gamma \in \Gamma \mid d(x_{0},\rho(\gamma)(x_{0})) \leqslant R\} \leqslant \#\{\gamma \in \Gamma \mid d_{S}(id,\gamma) \leqslant K(R+C)\}.$$
The right-hand side is finite since $\Gamma$ is finitely generated, this shows that $\Gamma$ is strongly discrete.

Since $$\ker\rho = \{\gamma\in\Gamma \mid \rho(\gamma) = id\} \subset \{\gamma\in\Gamma \mid d(x_{0},\rho(\gamma)(x_{0})) = 0\},$$
it follows that the kernel of $\rho$ is finite
\end{proof}

Let $\Gamma$ be a finitely generated group and $S$ be a finite generating set for $\Gamma$. We will denote by $\partial_{\infty}\Gamma$ the visual boundary of the Cayley graph $\Cay_{S}(\Gamma)$.

The proof of the stability of convex-cocompact representations of $\Gamma$ in $\H^{n}$ relies on the characterization of convex-cocompact representations as those whose orbit maps are quasi-isometric embeddings. We aim to get a similar result for representations in $\H^{\infty}$. The assumption that $\Gamma$ is finitely generated allows us to recover some compactness thanks to the following theorem from Mazur (see for example \cite[Theorem 2.8.15]{Megginson}). This fact was first observed in \cite[Lemma 4.2]{MP}.

\begin{theorem}[Mazur's compactness theorem]\label{thm:Mazur compactness theorem}
The closed convex hull of a compact subset of a Banach space is itself compact.
\end{theorem}

We now proceed to proving Theorem \ref{thm:QI<=>convex-cocompact} where the direct implication follows the proof of \cite[Proposition 4.3]{MP}.

\begin{proof}[Proof of Theorem \ref{thm:QI<=>convex-cocompact}]
Suppose first that the orbit map is a quasi-isometric embedding from $\Gamma$ to $\H^{\infty}$. It extends to a continuous boundary map $\partial \tau_{\rho}:\partial_{\infty}\Gamma \to \partial\H^{\infty}$ which is $\Gamma$-invariant (Proposition \ref{prop:extension_to_boundary}).

Fix a finite generating set $S$ for $\Gamma$. The Cayley graph $\Cay_{S}(\Gamma)$ is locally finite and the boundary $\partial_{\infty}\Gamma$ is compact (see for example \cite[Proposition 7.9]{GH}). Thus its image $\partial \tau_{\rho}(\partial_{\infty}\Gamma) \subset \partial \H^{\infty}$ is also compact for the cone topology on $\partial\H^{\infty}$. Denote by $\mathcal{C} \subset \H^{\infty}$ the closed convex hull of $\partial \tau_{\rho}(\partial_{\infty}\Gamma)$. With the Klein model of the hyperbolic space, $\H^{\infty}$ is identified with the unit ball $B$ of a real Hilbert space. The convex hull $\mathcal{C}$ is identified with the intersection of $B$ and the closed affine convex hull $\overline{\mathcal{C}}$ of $\partial \tau_{\rho}(\partial_{\infty}\Gamma)$ in $\overline{B}$. Theorem \ref{thm:Mazur compactness theorem} shows that $\overline{\mathcal{C}}$ is compact in $\overline{B}$. Therefore, $\mathcal{C}$ is locally compact and closed in $B$, so it is a closed and proper subspace of $\H^{\infty}$. Moreover, it follows from Proposition \ref{prop:QI=>strongly discrete+almost faithful} that the action of $\Gamma$ on $\mathcal{C}$ is strongly discrete.

In order to show the cocompactness of the action of $\Gamma$ on $\mathcal{C}$, it is then sufficient to show coboundedness since $\mathcal{C}$ is proper. Let $x_{0} \in \mathcal{C}$ and suppose for a contradiction that $\mathcal{C}$ contains a sequence $(y_{i})_{i\in\N}$ such that the distance $d(y_{i},\Gamma\cdot x_{0})$ goes to infinity. By strong discreteness of the action, we can suppose that $d(y_{i},\Gamma\cdot x_{0}) = d(y_{i},x_{0})$ up to replacing $y_{i}$ by one of its translates by $\Gamma$. Since $\overline{\mathcal{C}}$ is compact, $(y_{i})_{i\in\N}$ converges to a point $y \in \partial\mathcal{C} = \partial \tau_{\rho}(\partial_{\infty}\Gamma)$ up to picking a subsequence. Choose a sequence $(z_{i})_{i\in\N}$ in $\Cay_{S}(\Gamma)$ such that $\tau_{\rho}(z_{i})$ converges also to $y$. Thus the Gromov product
$$L_{ij} = (\tau_{\rho}(z_{i})\cdot y_{j})_{x_{0}} = \dfrac{1}{2}\left(d(\tau_{\rho}(z_{i}),x_{0})+d(y_{j},x_{0}) - d(\tau_{\rho}(z_{i}),y_{j})\right)$$
tends to infinity as $i,j \to \infty$. Let $m_{ij}$ be the point of the geodesic $[x_{0},y_{j}]$ at distance $L_{ij}$ from $x_{0}$. Since $d(y_{j},\Gamma\cdot x_{0}) = d(y_{j},x_{0}) \underset{j\to\infty}{\longrightarrow}+\infty$, we have $d(m_{ij},\Gamma\cdot x_{0}) = d(m_{ij},x_{0}) = L_{ij} \underset{i,j\to\infty}{\longrightarrow} +\infty$. By $\delta$-hyperbolicity, the distance from $m_{ij}$ to the geodesic $[x_{0},\tau_{\rho}(z_{i})]$ is bounded by some constant depending only on $\delta$. Since $\tau_{\rho}$ is a quasi-isometric embedding, the image of the geodesic $[e,z_{i}]$ in $\Cay_{S}(\Gamma)$ by $\tau_{\rho}$ is a quasi-geodesic between $x_{0}$ and $\tau_{\rho}(z_{i})$. Thus, by stability of quasi-geodesics (\cite[Theorem III.1.7]{BH}), the geodesic $[x_{0},\tau_{\rho}(z_{i})]$ is in the $R$-neighbourhood of $\tau_{\rho}(\Cay_{S}(\Gamma))$ for some fixed constant $R \geqslant 0$. Then the $m_{ij}$ remain at bounded distance from $\tau_{\rho}(\Cay_{S}(\Gamma))$ by triangle inequality so they also remain at bounded distance from $\tau_{\rho}(\Gamma)$ which is the orbit of $x_{0}$. This is in contradiction with $d(m_{ij},\Gamma\cdot x_{0}) \underset{i,j\to\infty}{\longrightarrow} +\infty$.

For the converse implication, let $\mathcal{C}$ be a convex and locally compact $\Gamma$-invariant subset of $\H^{\infty}$ on which $\Gamma$ acts properly and cocompactly. Then by Theorem \ref{thm:svarc-milnor version DSU}, the orbit map $\tau_{\rho} : \Gamma \to \mathcal{C}, \gamma \mapsto \gamma(x_{0})$ is a $\Gamma$--equivariant quasi-isometric embedding. And since the embedding of $\mathcal{C}$ to $\H^{\infty}$ is isometric, we get a quasi-isometric embedding $\tau_{\rho}:\Gamma \to \H^{\infty}$.
\end{proof}

\subsection{Stability of convex-cocompact representations}

Let $\Gamma$ be a finitely generated group. Let $S = \{\gamma_{1},\dots,\gamma_{m}\}$ be a generating family of $\Gamma$. Denote by $d_{S}$ the metric on the Cayley graph of $\Gamma$ associated to $S$. Given a representation $\rho:\Gamma\to \Isom(\H^{\infty})$ and a base point $x_{0}\in \H^{\infty}$, the orbit map $\tau_{\rho}:\Gamma \to \H^{\infty}$ given by
$$\tau_{\rho}(\gamma) = \rho(\gamma)(x_{0})$$
is $\Gamma$-equivariant. We will call $\rho$ a \textit{quasi-isometric representation} if its orbit map $\tau_{\rho}$ is a quasi-isometric embedding. 

Since $\H^{\infty}$ is $\delta$-hyperbolic, the group $\Gamma$ is Gromov-hyperbolic when $\rho:\Gamma\to\Isom(\H^{\infty})$ is a quasi-isometric representation.

\begin{theorem}[Local-to-global principle, \cite{CDP} Theorem $3.1.4$]\label{local-to-global}
Given $K \geqslant 1$, $C \geqslant 0$ and $\delta \geqslant 0$, there exist $\widehat{K}$, $\widehat{C}$ and $A$ so that if $J$ is an interval in $\R$, $X$ is a $\delta$-hyperbolic geodesic space and $h:J\to X$ is a $(K,C)$--quasi-isometric embedding when restricted to any connected subsegment of $J$ with length at most $A$, then $h$ is a $(\widehat{K},\widehat{C})$--quasi-isometric embedding.
\end{theorem}

Let $X$ be a $\delta$-hyperbolic geodesic metric space. The space of representations $\Hom (\Gamma,\Isom(X))$ can be embedded into $\Isom(X)^{m}$ by
$$\begin{array}{ccc}
    \Hom(\Gamma,\Isom(X)) & \to & \Isom(X)^{m}\\
    \rho & \mapsto & (\rho(\gamma_{1}),\dots,\rho(\gamma_{m})). 
\end{array}$$
Endow $\Isom(X)$ with the pointwise convergence topology. If $f\in \Isom(X)$, a basis of neighbourhoods of $f$ is given by the open sets of the form $$V(f,\epsilon,P) = \{g\in \Isom(X)\mid \forall x\in P\quad d(f(x),g(x)) < \epsilon\}$$
where $P$ is a finite set of points in $X$ and $\epsilon>0$. The topology on $\Hom(\Gamma,\Isom(X))$ is then inherited from the product topology on $\Isom(X)^{m}$.\\

The following theorem is due to Marden in the case of representations into $\Isom(\H^{3})$ and to Thurston for $\Isom(\H^{n})$, $n\in\N$. The same proof remains valid for any $\delta$-hyperbolic space. We follow the proof of \cite[Theorem 11.4]{Canary}.

\begin{theorem}
Let $X$ be a geodesic metric space. If $\Gamma$ is a finitely generated group and $\rho:\Gamma\to \Isom(X)$ is a quasi-isometric representation, then there exists a neighbourhood $U$ of $\rho$ in $\Hom(\Gamma,\Isom(X))$ such that if $\sigma\in U$, then $\sigma$ is a quasi-isometric representation. Moreover, we may choose $U$ so that all the representations $\sigma \in U$ have the same constants of quasi-isometry.
\end{theorem}

\begin{proof}
Let $x_{0}\in X$ and suppose that the orbit map $\tau_{\rho}$ is a $(K,C)$--quasi-isometric embedding with respect to the finite generating set $S = \{\gamma_{1},\dots,\gamma_{m}\}$. Let $A > 0$ to be defined later. Let $U$ be an open neighbourhood of $\rho$ in $\Hom(\Gamma,\Isom(X))$ such that for all $\sigma \in U$ and $\gamma\in\Gamma$ with $d_{S}(id,\gamma) \leqslant A$, $d(\rho(\gamma)(x_{0}),\sigma(\gamma)(x_{0}))< 1$. It suffices to choose 
$$U = \prod_{i=1}^{m} V(\rho(\gamma_{i}),A+1,P)$$
where $P = \{\rho(\gamma)(x_{0}) \mid \gamma\in\Gamma\ \text{and}\ d_{S}(id,\gamma)\leqslant A\}$ is finite since $\#\{\gamma\in\Gamma \mid d_{S}(id,\gamma) \leqslant A\}<\infty$.\\
Since $\tau_{\rho}$ is a $(K,C)$--quasi-isometric embedding, then for all $\gamma \in \Gamma$,
$$\dfrac{1}{K}d_{S}(id,\gamma) - C \leqslant d(x_{0},\tau_{\rho}(\gamma)) \leqslant Kd_{S}(id,\gamma) + C.$$
Let $\sigma\in U$, we see that if $d_{S}(id,\gamma) \leqslant A$, then
$$\dfrac{1}{K}d_{S}(id,\gamma) - C - 1\leqslant d(x_{0},\tau_{\sigma}(\gamma)) \leqslant Kd_{S}(id,\gamma) + C + 1$$
and by $\Gamma$-equivariance of $\tau_{\sigma}$,
$$\dfrac{1}{K}d_{S}(\alpha,\beta) - C - 1\leqslant d(\tau_{\sigma}(\alpha),\tau_{\sigma}(\beta)) \leqslant Kd_{S}(\alpha,\beta) + C + 1$$
for all $\alpha,\beta \in \Gamma$ with $d_{S}(\alpha,\beta)\leqslant A$.

If $\Cay_{S}(\Gamma)$ is the Cayley graph of $\Gamma$ with respect to $S$, then we may extend $\tau_{\sigma}$ to $\Cay_{S}(\Gamma)$ by mapping all the edges to geodesics in $X$. The resulting map is a $(K',C')$--quasi-isometric embedding on all geodesic segments in $\Cay_{S}(\Gamma)$ of length at most $A$ for some $K' \geqslant 1$ and $C' \geqslant 0$ depending only on $K$ and $C$. We may then choose $A$ according to Theorem \ref{local-to-global}. This theorem provides some $\widehat{K}$ and $\widehat{C}$ such that $\tau_{\sigma}$ is a $(\widehat{K},\widehat{C})$--quasi-isometric embedding on all geodesic segments in $\Cay_{S}(\Gamma)$. Thus $\tau_{\sigma}$ is a $(\widehat{K},\widehat{C})$--quasi-isometric embedding on the whole Cayley graph $\Cay_{S}(\Gamma)$.
\end{proof}

\subsection{Convex-coboundedness and convex-cocompactness}

The argument of Theorem \ref{thm:QI<=>convex-cocompact} to get a locally compact convex set $\mathcal{C} \subset \H^{\infty}$ on which the group $\Gamma$ acts cocompactly fails when the Cayley graph of $\Gamma$ is not locally finite. The assumption that $\Gamma$ is finitely generated is then important for the proof to hold. To illustrate this, we will consider embeddings of trees into $\H^{\infty}$ introduced in \cite{BIM} to produce a representation of an infinitely generated group into $\Isom(\H^{\infty})$ which is convex-cobounded but not convex-cocompact even though this embedding is quasi-isometric.

\begin{theorem}[\cite{BIM}, Theorems $A$ and $8.1$]\label{thm:BIM}
Let $\mathcal{T}$ be a simplicial tree, denote by $V$ the set of vertices of $\mathcal{T}$ and by $\H$ the hyperbolic space of dimension $|V|-1$, $\H = \H^{|V|-1}$. Let $d_{\mathcal{T}}$ be the metric on $\mathcal{T}$ that assigns length $1$ to all the edges of $\mathcal{T}$. Then for every $\lambda > 1$ there exists an embedding $\psi : \mathcal{T} \to \H$ and a representation $\rho : \Aut(\mathcal{T}) \to \Isom(\H)$ such that
\begin{enumerate}
    \item the map $\psi$ is $\Aut(\mathcal{T})$-equivariant for $\rho$.
    \item $\lambda^{d_{\mathcal{T}}(x,y)} = \cosh d(\psi(x),\psi(y))$ for any $x,y \in V$.
    \item $\psi$ extends to an equivariant boundary map $\partial \psi : \partial \mathcal{T} \to \partial \H$ which is a homeomorphism onto its image.
    \item $\psi(V)$ is cobounded in the convex hull of the image of $\partial\psi$.
\end{enumerate}
\end{theorem}

\begin{remark}
The second point of the theorem implies that $\psi$ is a $(K,C)$--quasi-isometric embedding on the vertices of $\mathcal{T}$ for $K = \max\left\{\ln(\lambda),\dfrac{1}{\ln(\lambda)}\right\}$ and $C = \dfrac{\ln 2}{\ln \lambda}$.
\end{remark}

Let $\mathcal{T} = (V,E)$ be the regular tree of countably infinite valency, let $d_{\mathcal{T}}$ be the combinatorial distance on $\mathcal{T}$ (every edge has length $1$). The group $\Aut(\mathcal{T})$ acts cocompactly on $\mathcal{T}$ and a fundamental region is a half-edge of $\mathcal{T}$, \ie an edge of the first barycentric subdivision of $\mathcal{T}$. The free group $\F_{\infty}$ is a subgroup of $\Aut(\mathcal{T})$ and a fundamental region for the action $\F_{\infty} \action \mathcal{T}$ is given by the tree $\mathcal{T}'$ corresponding to the ball of radius $\dfrac{1}{2}$ around a vertex, \ie a vertex of $\mathcal{T}$ and all its adjacent edges in the first barycentric subdivision of $\mathcal{T}$.

\begin{proposition}\label{prop:cobounded action on tree}
The action of $\F_{\infty}$ on $\mathcal{T}$ is cobounded but not cocompact.
\end{proposition}

\begin{proof}
Let $S = \{s_{1},s_{2},\dots\}$ be a system of generators of $\F_{\infty}$. The tree $\mathcal{T}$ can be seen as the Cayley graph of $\F_{\infty}$ associated to $S$ where all the edges around each vertex are labelled with one element of $S \cup S^{-1}$ (two edges starting from the same vertex cannot have the same label). The free group $\F_{\infty}$ acts on $\mathcal{T}$ by sending an edge onto another edge with the same label. This action is also transitive on the vertices of $\mathcal{T}$, so the quotient $\widehat{\mathcal{T}} = \mathcal{T}/\F_{\infty}$ is isomorphic to a single vertex with infinitely many loops of length $1$ around it, hence it is bounded. To show that it is not compact, let $v$ be a vertex of $\mathcal{T}$ and for all $i\in\N$, let $m_{i}$ be the point at distance $\dfrac{1}{2}$ from $v$ on the edge labelled by $s_{i}$. The distance $d_{\mathcal{T}}$ on $\mathcal{T}$ induces a distance $d_{\widehat{\mathcal{T}}}$ on the quotient such that for all $i \neq j \in \N$, $d_{\widehat{\mathcal{T}}}(m_{i},m_{j}) = 1$. Thus the sequence $(m_{i})_{i\in\N}$ has no convergent subsequence and $\widehat{\mathcal{T}}$ is not compact.
\end{proof}

Fix $\lambda > 1$. By Theorem \ref{thm:BIM}, there is a representation $\rho = \rho_{\lambda}:\F_{\infty} < \Aut(\mathcal{T}) \to \Isom(\H^{\infty})$ and an $\Aut(\mathcal{T})$--equivariant quasi-isometric embedding $\psi = \psi_{\lambda}:\mathcal{T} \to \H^{\infty}$ which extends to the boundary $\partial\mathcal{T}$.

Denote by $\mathcal{C}_{\lambda}$ the closed convex hull of $\partial\psi(\partial\mathcal{T})$ in $\H^{\infty}$ and let $\mathcal{C}_{0}$ be the closed convex hull of $\psi(\mathcal{T})$. Both $\mathcal{C}_{\lambda}$ and $\mathcal{C}_{0}$ are $\Aut(\mathcal{T})$--invariant and $\mathcal{C}_{0}$ is the unique minimal invariant closed convex set in $\H^{\infty}$, thus we have $\mathcal{C}_{0} \subset \mathcal{C}_{\lambda}$. However they are not equal in general, see \cite[Section 5.A]{MP}.

\begin{proposition}
The group $\rho(\F_{\infty}) < \Isom(\H^{\infty})$ is convex-cobounded.
\end{proposition}

\begin{proof}
The embedding $\psi: \mathcal{T}\to \H^{\infty}$ being $\Aut(\mathcal{T})$--equivariant, we have $\psi(\mathcal{T}) = \rho(\F_{\infty}) \cdot \psi(\mathcal{T}')$, where $\mathcal{T}'$ is a bounded fundamental region for $\F_{\infty} \action \mathcal{T}$. By the fourth point of Theorem \ref{thm:BIM}, $\psi(\mathcal{T})$ is cobounded in $\mathcal{C}_{\lambda}$, so $\rho(\F_{\infty})$ acts coboundedly on $\mathcal{C}_{\lambda}$ since $\psi(\mathcal{T}')$ is bounded.
\end{proof}

\begin{lemma}\label{lem:cocompact on minimal convex set}
Let $G < \Aut(\mathcal{T})$. If $\rho(G)$ acts cocompactly on any invariant closed convex subset $\mathcal{C} \subset \H^{\infty}$, then the action of $G$ on the minimal invariant convex set $\mathcal{C}_{0}$ via $\rho$ is also cocompact.
\end{lemma}

\begin{proof}
Suppose that $\mathcal{C}= \rho(G)\cdot A$, where $A \subset \H^{\infty}$ is compact. Let $x \in \mathcal{C}_{0}$. Since $\mathcal{C}_{0} \subset \mathcal{C}$, there exists $g \in G$ such that $x \in \rho(g)A$. Then $x \in \rho(g)A \cap \mathcal{C}_{0} = \rho(g)(A \cap \mathcal{C}_{0})$ by invariance of $\mathcal{C}_{0}$. Therefore $\mathcal{C}_{0} = \rho(G)\cdot \left(A \cap \mathcal{C}_{0}\right)$. Moreover, $A \cap \mathcal{C}_{0}$ is closed in $A$, so $A \cap \mathcal{C}_{0}$ is compact.
\end{proof}

This lemma implies that in order to prove that $\rho(\F_{\infty}) < \Isom(\H^{\infty})$ is not convex-cocompact, it is sufficient to show this property for the action of $\F_{\infty}$ on $\mathcal{C}_{0}$ via $\rho$. 

\begin{proposition}\label{prop:non-cocompact action on tree}
    The action of $\rho(\F_{\infty})$ on $\psi(\mathcal{T})$ is not cocompact.
\end{proposition}

\begin{proof}
The proof is similar to that of Proposition \ref{prop:cobounded action on tree}. Suppose that $\psi(\mathcal{T}) \subset \rho(\F_{\infty})\cdot K$ for $K \subset \H^{\infty}$. Let $v$ be a vertex in $\mathcal{T}$ and enumerate the edges incident to $v$. For all $i \in \N$, denote by $m_{i}$ the point on the $i$--th edge around $v$ such that $d_{\mathcal{T}}(v,m_{i}) = \dfrac{1}{2}$. Then for all $i\in \N$, there exists $m_{i}'$ in the orbit $\F_{\infty}\cdot m_{i}$ such that $\psi(m_{i}') \in K$. For $i\neq j$, we have $d_{\mathcal{T}}(m_{i}',m_{j}') \geqslant 1$ so by the second point of Theorem \ref{thm:BIM}, 
$$d(\psi(m_{i}'),\psi(m_{j}')) = \cosh^{-1}\left(\lambda^{d_{\mathcal{T}}(m_{i}',m_{j}')}\right) \geqslant \cosh^{-1}(\lambda) > 0.$$
The sequence $\left(\psi(m_{i}')\right)_{i\in\N} \in K^{\N}$ has no convergent subsequence, therefore $K$ cannot be compact.
\end{proof}

\begin{corollary}
The group $\rho(\F_{\infty}) < \Isom(\H^{\infty})$ is not convex-cocompact.
\end{corollary}

\begin{proof}
By Lemma \ref{lem:cocompact on minimal convex set}, it suffices to prove that $\rho(\F_{\infty})$ does not act cocompactly on $\mathcal{C}_{0}$. If $\mathcal{C}_{0} \subset \rho(\F_{\infty})\cdot K$ for some $K \subset \H^{\infty}$, then we may construct a sequence of points $m_{i}' \in \mathcal{T}$ such that $\psi(m_{i}') \in K$ and all the $d(\psi(m_{i}'),\psi(m_{j}'))$ are greater than a uniform positive constant for $i\neq j$, as in Proposition \ref{prop:non-cocompact action on tree}. Since $\psi(\mathcal{T}) \subset \mathcal{C}_{0}$, all the $\psi(m_{i}')$ also belong to $\mathcal{C}_{0}$.
\end{proof}

\section{Embeddings of \texorpdfstring{$\Isom(\H^{2})$}{Isom(H^2)} into \texorpdfstring{$\Isom(\H^{\infty})$}{Isom(H^infty)}}

\begin{definition}
Let $\rho:\Gamma \to \Isom(\H^{\infty})$ be a representation of a group $\Gamma$. We say that $\rho$ is \emph{irreducible} if $\rho(\Gamma)$ has no fixed point in the boundary $\partial \H^{\infty}$ and does not preserve any non-trivial closed totally geodesic subspace of $\H^{\infty}$.
\end{definition}

\begin{remark}
This definition coincides with the usual definition of irreducibility for the linear representation of $\Gamma$ on the underlying Hilbert space $\mathcal{H}$: $\rho$ is irreducible if and only if there are no non-trivial invariant closed linear subspaces in $\mathcal{H}$. This was also referred to as \emph{geometric Zariski density} in \cite[Section 5.B]{MP}.
\end{remark}

From the hyperboloid model $\H^{\infty} = \{x \in \mathcal{H} \mid Q(x) = -1, x_{0} > 0\}$, we can obtain an embedding of $\H^{n}$ into $\H^{\infty}$ by restricting to vectors $x = (x_{i})_{i\in\N} \in \mathcal{H}$ where all the coordinates vanish for $i > n+1$. This also induces an injection of the isometry group $\Isom(\H^{n}) \hookrightarrow \Isom(\H^{\infty})$. Monod and Py classified all the continuous representations from $\Isom(\H^{n})$ to $\Isom(\H^{\infty})$ in \cite{MP} that are irreducible.

\begin{theorem}[\cite{MP}, Theorem $B$] \label{thm:exotic deformation B}
\begin{enumerate}
    \item Let $\rho:\Isom(\H^{n}) \to \Isom(\H^{\infty})$ be a continuous non-elementary action. There exists $t\in (0,1]$ such that $\ell_{\H^{\infty}}(\rho(g)) = t\ell_{\H^{n}}(g)$ for all $g\in \Isom(\H^{n})$, where $\ell_{\H^{\infty}}$ and $\ell_{\H^{n}}$ denote the translation lengths in $\H^{\infty}$ and $ \H^{n}$ respectively. Moreover, $t=1$ if and only if $\rho$ preserves an $n$-dimensional totally geodesic subspace of $\H^{\infty}$.
    \item For each $t\in (0,1)$, there is, up to conjugacy in $\Isom(\H^{\infty})$, exactly one irreducible continuous representation $\rho_{t}:\Isom(\H^{n}) \to \Isom(\H^{\infty})$ such that for all $g\in\Isom(\H^{n})$, $\ell_{\H^{\infty}}(\rho_{t}(g)) = t\ell_{\H^{n}}(g)$.
\end{enumerate}
\end{theorem}

\begin{theorem}[\cite{MP}, Theorem C]\label{thm:exotic deformation C}
There exists a smooth harmonic $\rho_{t}$-equivariant map $f_{t}:\H^{n}\to\H^{\infty}$ which is asymptotically an isometric embedding after rescaling, \ie there is $D \geqslant 0$ such that for all $x,y \in \H^{n}$,
$$\left|d_{\H^{\infty}}(f_{t}(x),f_{t}(y)) - td_{\H^{n}}(x,y)\right| \leqslant D.$$
Moreover, $f_{t}$ extends to a continuous map $f_{t}:\overline{\H^{n}}\to\overline{\H^{\infty}}$.
\end{theorem}

Notice that being asymptotically an isometric embedding after rescaling implies that $f_{t}$ is a quasi-isometric embedding from $\H^{n}$ to $\H^{\infty}$.

For $t\in (0,1)$, since $\rho_{t}$ is non-elementary, there is a unique minimal non-empty closed convex $\rho_{t}$-invariant subset $\mathcal{C}_{t} \subset \H^{\infty}$ (see \cite{monod}).

\begin{theorem}[\cite{MP}, Theorem D and Proposition E]\label{thm:exotic deformation DE}
For any $t\in (0,1]$, the $\CAT(-1)$ space $\mathcal{C}_{t}$ is proper, the action of $\Isom(\H^{n})$ on $\mathcal{C}_{t}$ is cocompact and $\rho_{t}$ induces an isomorphism $\Isom(\H^{n}) \simeq \Isom(\mathcal{C}_{t})$.
\end{theorem}

Notice that if $\Gamma$ is a cocompact lattice of $\Isom(\H^{n})$ and if $\rho_{t}:\Isom(\H^{n})\to \Isom(\H^{\infty})$ is a representation for some $t\in (0,1]$ as in Theorem \ref{thm:exotic deformation B}, the restriction of $\rho_{t}$ to $\Gamma$ is then convex-cocompact. Indeed, since $\Isom(\H^{n}) \curvearrowright \mathcal{C}_{t}$ is cocompact, there is a compact set $K_{1}\subset\mathcal{C}_{t}$ such that $\mathcal{C}_{t} = \Isom(\H^{n})\cdot K_{1}$. Moreover, since $\Gamma$ is a cocompact lattice in $\Isom(\H^{n})$, we have $\Isom(\H^{n}) = \Gamma\cdot K_{2}$ for some compact subset $K_{2} \subset \Isom(\H^{n})$. Thus, $\mathcal{C}_{t} = \Isom(\H^{n})\cdot K_{1} = (\Gamma\cdot K_{2})\cdot K_{1} = \Gamma\cdot K$, where $K = K_{2}\cdot K_{1}$ is a compact subset of $\mathcal{C}_{t}$. Using Corollary \ref{cor:openness thm}, we can find a neighbourhood $U$ of $\rho_{t|\Gamma}$ in $\Hom(\Gamma, \Isom(\H^{\infty}))$ consisting only of convex-cocompact representations of $\Gamma$.

\subsection{Some properties of the exotic representations}

Fix $t \in (0,1]$ and let $\rho_{t} : \Isom(\H^{n}) \to \Isom(\H^{\infty})$ and $f_{t}:\H^{n}\to\H^{\infty}$ be a representation and its corresponding harmonic map given by Theorems \ref{thm:exotic deformation B} and \ref{thm:exotic deformation C} respectively.

\begin{lemma}[\cite{MP}, Lemma $4.1$]
    Let $\mathcal{C}_{t}$ be the closed convex hull of $f_{t}(\H^{n})$ inside $\H^{\infty}$. The action of $\Isom(\H^{n})$ on $\mathcal{C}_{t}$ is minimal, \ie $\mathcal{C}_{t}$ admits no nontrivial $\Isom(\H^{n})$-invariant closed convex subspace. Moreover, this is the unique minimal $\Isom(\H^{n})$-invariant closed convex subset of $\H^{\infty}$. 
\end{lemma}

We point out in the next lemma some properties of $\rho_{t}$ and $f_{t}$ that can be deduced from the results in \cite{MP}. 

In the following, a map $f:X \to \R$, where $X$ is a complete geodesic metric space, will be called \emph{analytic along geodesics} if for any geodesic line $L$ (the image of an isometric embedding $\gamma:\R \to X$), the restriction $f|_{L} = f\circ \gamma : \R \to \R$ is an analytic function.

\begin{lemma}\label{lem:some properties}
Let $0<t\leqslant 1$,
\begin{enumerate}
    \item If $\Gamma < \Isom(\H^{n})$ is a cocompact lattice, then $\rho_{t}(\Gamma)$ is cocompact in $\rho_{t}(\Isom(\H^{n})) \simeq \Isom(\mathcal{C}_{t})$. 
    \item If $\Gamma < \Isom(\H^{n})$ has no fixed point in $\partial \H^{n}$, then $\rho_{t}(\Gamma)$ also has no fixed point in $\partial \mathcal{C}_{t} = f_{t}(\partial\H^{n})$.
    \item The harmonic map $f_{t}: \H^{n}\to\H^{\infty}$ is analytic and therefore, if $Y$ is a totally geodesic subspace of $\H^{\infty}$, the map $x\mapsto d(f_{t}(x),Y)$ is analytic along geodesics.
\end{enumerate}
\end{lemma}

\begin{proof}
The first point is a consequence of Theorem \ref{thm:exotic deformation DE} and the proof of Theorem \ref{thm:QI<=>convex-cocompact}, where the invariant convex set is $\mathcal{C}_{t}$.

The second point follows from the equivariance of $f_{t}$ and the fact that $f_{t}$ is a quasi-isometric embedding which induces an injective map (actually a homeomorphism) from $\partial \H^{n}$ to $f(\partial \H^{n})$ by Proposition \ref{prop:extension_to_boundary}.

For the last point, the analyticity of the harmonic map $f_{t}$ comes from its expression explicited in \cite[section 3.A]{MP}. For $g \in \Isom(\H^{n})$ and $b \in \partial \H^{n}$, $f_{t}$ can be identified with the map from $\H^{n}$ to $\L^{2}(\partial \H^{n},\R)$ defined by
$$f_{t}(g\cdot o)(b) = |\Jac(g^{-1})(b)|^{1+\frac{t}{n-1}} = \left(\dfrac{B_{n}(o,b)}{B_{n}(g(o),b)}\right)^{t+n-1},$$
where $B_{n}$ is the bilinear form on $\R^{n+1}$ associated to the quadratic form of signature $(n,1)$ defining $\H^{n}$ and $o \in \H^{n}$ is some base point.
Define $\varphi:\H^{\infty} \to \R$ by $\varphi(x) = d(x,Y)$. If $\pi_{Y}$ denotes the orthogonal projection onto $Y$, then $\varphi(x) = d(x,\pi_{Y}(x))$. Since $Y \subset \H^{\infty}$ is totally geodesic, it is a hyperbolic subspace of $\H^{\infty} \subset \mathcal{H}$ and we may choose a basis $(e_{i})_{i\in\N}$ of the Hilbert space $\mathcal{H}$ such that 
$$Y = \{x \in \H^{\infty} \mid \forall i\in I\quad x_{i} = 0\},$$
where $I$ is some subset of $\N_{\geqslant 1} = \{1,2,\dots\}$ and the $x_{i}$'s are the coordinates of $x$ in this basis. Then for $x \in \H^{\infty}$, $\pi_{Y}(x)= \dfrac{y}{\sqrt{-Q(y)}}$ with 
$y = x-\sum_{i\in I} B(x,e_{i})e_{i} = \sum_{i\notin I}x_{i}e_{i}$. We then get 
$$\cosh{d(x,Y)} = \cosh{d(x,\pi_{Y}(x))} = -B(x,\pi_{Y}(x)) = \left(x_{0}^{2}-\sum_{i\in \N_{\geqslant 1}\setminus I} x_{i}^{2}\right)^{\frac{1}{2}}.$$
Let $\hat{J}$ denote the bounded linear map on $\mathcal{H}$ such that $\hat{J}e_{0} = e_{0}$, $\hat{J}e_{i} = 0$ if $i \in I$ and $\hat{J}e_{i} = -e_{i}$ otherwise. Then $$\varphi(x) = \cosh^{-1}\left(\langle x,\hat{J}x\rangle\right)^{\frac{1}{2}} = \cosh^{-1}\left(\transp{x}\hat{J}x\right)^{\frac{1}{2}}.$$
For any hyperbolic geodesic $L \subset \H^{n}$, there exist $a,b \in \R^{n+1}$ such that $L$ is the image of $\gamma:u\in\R \mapsto \cosh(u)a+\sinh(u)b \in \H^{n}$. The restriction of $\varphi\circ f_{t}$ to $L$ gives
\begin{align*}
    (\varphi\circ f_{t})_{|L}(u) &= (\varphi\circ f_{t})(\gamma(u))\\
    &= \varphi(f_{t}(\gamma(u)))\\
    &= \cosh^{-1}\left(\transp{f_{t}(\gamma(u))}\hat{J}f_{t}(\gamma(u))\right)^{\frac{1}{2}}
\end{align*}
which is an analytic function in $u \in \R$.
\end{proof}

Caprace and Monod proved the following theorem which can be thought of as an analog of the Borel-density theorem for lattices acting a proper $\CAT(0)$ space.

\begin{theorem}[\cite{Caprace-Monod:subgroups}, Theorem 2.4]\label{thm:caprace-monod}
Let $G$ be a locally compact group with a continuous isometric action on a proper $\CAT(0)$ space $X$ without Euclidean factor. If $G$ acts minimally on $X$ and without a
global fixed point in $\partial X$, then any closed subgroup with finite invariant covolume in $G$ still has these properties.
\end{theorem}

In our context, we are interested in the case of the group $G=\rho_{t}(\Isom(\H^{n}))$ acting continuously by isometries on the space $X = \H^{\infty}$. Instead of a closed subgroup in $\Isom(\H^{n})$ with finite invariant covolume, we consider a convex-cocompact subgroup $A < \Isom(\H^{n})$ and its image in $\Isom(\H^{\infty})$ by an irreducible representation $\rho_{t}$ ($0<t<1$). In this case, it remains at least true that $\rho_{t}(A)$ admits no global fixed point in the boundary of $\partial \H^{\infty}$ as soon as $A < \Isom(\H^{n})$ has no fixed point in $\H^{n}\cup\partial\H^{n}$. A similar statement was proved by Caprace and Monod as a first step to obtain Theorem \ref{thm:caprace-monod}, see \cite[Proposition 2.1]{Caprace-Monod:subgroups}. It relies in particular on the following result asserting that the intersection of nested sequences of closed convex sets in hyperbolic spaces is either non-empty, or it converges to one point in the boundary.

\begin{theorem}[\cite{CL}, Theorem $1.1$] \label{thm:nested closed convex sets}
    Let $X$ be a complete $\CAT(0)$ space of finite telescopic dimension and $\{X_{\alpha}\}_{\alpha\in A}$ be a filtering family of closed convex subspaces. Then either the intersection $\bigcap_{\alpha\in A}X_{\alpha}$ is non-empty, or the intersection of the visual boundaries $\bigcap_{\alpha\in A}\partial X_{\alpha}$ is a non-empty subset of $\partial X$ of intrinsic radius at most $\pi/2$.
\end{theorem}

See also \cite[Theorem 4.3]{Duc13} for a more geometric proof of this statement. In particular, when $X = \H^{\infty}$, if the intersection of the visual boundaries is non-empty, then it is reduced to a point $\{\xi\}$.

\begin{proposition}\label{prop:no fixed point in boundary}
Let $\Gamma < \Isom(\H^{n})$ be a subgroup which has no fixed point in $\H^{n} \cup \partial \H^{n}$. Then $\rho_{t}(\Gamma) < \rho_{t}(\Isom(\H^{n})) \simeq \Isom(\mathcal{C}_{t})$ has no fixed point in $\partial\H^{\infty}$, where $\mathcal{C}_{t}$ is the convex hull of $f_{t}(\H^{n})$ inside $\H^{\infty}$.
\end{proposition}

\begin{proof}
Suppose for a contradiction that $\rho_{t}(\Gamma)$ has a fixed point $\xi\in \partial\H^{\infty}$. Fix a base point $x_{0}\in \mathcal{C}_{t}$. The Busemann function restricted to $\mathcal{C}_{t}$, $\beta_{\xi,x_{0}}|_{\mathcal{C}_{t}}$, is convex and does not have a minimum on $\mathcal{C}_{t}$. 

Indeed, if the minimum set $\Min(\beta_{\xi,x_{0}}|_{\mathcal{C}_{t}}) \subset \mathcal{C}_{t}$ is non-empty, it should lie inside some horosphere of $\H^{\infty}$ centered at $\xi$. It is moreover bounded, otherwise its closure in $\overline{\H^{\infty}}$ would contain $\xi$, and we would have $\xi \in \partial \mathcal{C}_{t} = \partial(f_{t}(\H^{n}))$ which is in contradiction with the second point of Lemma \ref{lem:some properties}. Since $\mathcal{C}_{t}$ is at finite distance from $f_{t}(\H^{n})$, there exists $K \in \R$ such that the $\rho_{t}(\Gamma)$-invariant set
$$\{x \in f_{t}(\H^{n}) \mid d(x,\Min(\beta_{\xi,x_{0}}|_{\mathcal{C}_{t}}))\leqslant K\}$$
is non-empty. It is moreover bounded and its preimage by $f_{t}$ is a bounded subset of $\H^{n}$ which is $\Gamma$-invariant by $\rho_{t}$-equivariance of $f_{t}$. It then has a circumcentre which is a fixed point for $\Gamma$ in $\H^{n}$, this contradicts the hypothesis on $\Gamma$. So $\beta_{\xi,x_{0}}|_{\mathcal{C}_{t}}$ does not have a minimum. 

For $a > \inf(\beta_{\xi,x_{0}}|_{\mathcal{C}_{t}})$, define $$C_{a} = \beta_{\xi,x_{0}}|_{\mathcal{C}_{t}}^{-1}((-\infty,a)) \subset \mathcal{C}_{t}.$$
All the $C_{a}$, for $a > \inf(\beta_{\xi,x_{0}}|_{\mathcal{C}_{t}})$, are convex and thus $\bigcap_{a > \inf(\beta_{\xi,x_{0}}|_{\mathcal{C}_{t}})} C_{a} = \emptyset$. Since $\H^{\infty}$ is $\delta$-hyperbolic, its telescopic dimension is finite (it is actually equal to $1$), hence
$\bigcap_{a > \inf(\beta_{\xi,x_{0}}|_{\mathcal{C}_{t}})} \partial C_{a}$ is non-empty (by Theorem \ref{thm:nested closed convex sets}) and must contain $\xi$, so we deduce again that $\xi \in \partial \mathcal{C}_{t} = \partial(f_{t}(\H^{n}))$, which is impossible.
\end{proof}

For $X = \H^{n}$ or $\H^{\infty}$, the \emph{limit set} of a subgroup $\Gamma < \Isom(X)$ is $\Lambda(\Gamma) := \overline{\Gamma\compactcdot o}\cap \partial X$ for any base point $o \in X$. A subgroup $\Gamma < \Isom(X)$ is called \emph{non-elementary} if it has no fixed point in $X \cup \partial X$ and fixes no geodesic in $X$. Equivalently, $\Gamma$ is elementary if and only if its limit set is finite.

\begin{corollary}\label{cor:image of non-elementary group}
For $0<t\leqslant 1$, if $A < \Isom(\H^{n})$ is a discrete and non-elementary subgroup, then $\rho_{t}(A)$ is also non-elementary in $\Isom(\H^{\infty})$.
\end{corollary}

\begin{proof}
    The image $\rho_{t}(A)$ has no global fixed point in $\H^{\infty}$, otherwise $A$ would consist only of elliptic isometries since $\rho_{t}$ preserves the types of the isometries (see \cite[Proposition 2.1]{MP}). Moreover, by Proposition \ref{prop:no fixed point in boundary}, $\rho_{t}(A)$ also has no fixed point in $\partial \H^{\infty}$. And if $\rho_{t}(A)$ preserves a geodesic in $\H^{\infty}$ with endpoints $\xi,\eta \in \partial\H^{\infty}$, then it leaves the pair $\{\xi,\eta\}$ invariant. The limit set of $\rho_{t}(A)$ is then contained in $\{\xi,\eta\}$. This contradicts the fact that $A$ is non-elementary since the $\rho_{t}$-equivariant embedding $f_{t}$ extends to an injective map from the limit set of $A$ to that of $\rho_{t}(A)$.
\end{proof}

\subsection{Spectral representation for real Hilbert spaces}

Before discussing the deformations of convex-cocompact representations by bending, we first recall the spectral representation theorem for operators on a real Hilbert space that we will use in the next section to show that the spaces of deformations are infinite-dimensional. In this paragraph, $\mathcal{H}$ refers to a real Hilbert space and $\mathbf{B}(\mathcal{H})$ is the set of bounded operators on $\mathcal{H}$.

Let $T \in \mathbf{B}(\mathcal{H})$ be a bounded operator on $\mathcal{H}$. The \emph{adjoint} operator of $T$, denoted by $T^{*}$, is the (unique) bounded linear operator on $\mathcal{H}$ such that $\langle Tx,y\rangle = \langle x, T^{*}y\rangle$ for all $x,y \in \mathcal{H}$. The operator $T$ is said to be \emph{normal} if $TT^{*} = T^{*}T$.

In the finite-dimensional situation, every normal operator on a real Hilbert space is othogonally equivalent to a block-diagonal matrix of the form
$$\begin{pmatrix}
    \lambda_{1}& & & & \\
    & \lambda_{2} & & & \\
    & & \ddots & & & \\
    & & & B_{1} & & \\
    & & & & B_{2} & \\
    & & & & & \ddots
\end{pmatrix}$$
where the $\lambda_{k}$ are its real eigenvalues (counted with their multiplicities) and the $B_{i}$ are $2\times 2$ matrices of the form $\begin{pmatrix}
    a_{k} & -b_{k}\\
    b_{k} & a_{k}
\end{pmatrix}$ corresponding to the complex eigenvalues $a_{k} \pm ib_{k}$. There is an analog of this decomposition in the infinite-dimensional situation.

In \cite{Goodrich}, Goodrich obtained a similar description for operators on infinite-dimensional real Hilbert spaces where the blocks correspond to operators of the form given in the following example.

\begin{example}[\cite{Goodrich}]\label{ex:normal_operator}
Let $X \subset \C$ be a compact subset which is symmetric about the real axis. Let $\mu$ be a regular Borel measure on $X$ and consider the real Hilbert space $\mathcal{H}= \L^{2}(X,\mu)$ of real-valued square-integrable functions on $X$ with respect to the measure $\mu$. Suppose that $\mu$ is symmetric about the real axis, \ie $\mu(U) = \mu(U^{*})$ where $U^{*}$ is the reflection of 
$U$ about the real axis. 
Let $\mathcal{H}_{e}$ be the set of functions in $\mathcal{H}$ that are symmetric about the real axis and $\mathcal{H}_{o}$ those that are anti-symmetric (the subscripts refer to "even" and "odd"). We have $\mathcal{H} = \mathcal{H}_{e} \oplus \mathcal{H}_{o}$. 

Define the operator $T$ on $\mathcal{H}$ by 
$$Tf = \begin{pmatrix}
    M_{\Re} & -M_{\Im}\\
    M_{\Im} & M_{\Re}
\end{pmatrix}\begin{pmatrix}
    f_{e}\\
    f_{o}
\end{pmatrix},$$
where $f = f_{e} + f_{o} \in \mathcal{H}$, $\Re,\Im: \C \to \R$ denote the real an imaginary parts of a complex number and $M_{\varphi}:\L^{2}(\C) \to \L^{2}(\C)$ is the operator defined by
$$M_{\varphi}(f)(z) = \varphi(z)f(z)$$
for a function $\varphi: \C \to \C$. Then $T$ is a normal operator on $\mathcal{H}$ and its adjoint is
$$T^{*}f = \transp{T}f = \begin{pmatrix}
    M_{\Re} & M_{\Im}\\
    -M_{\Im} & M_{\Re}
\end{pmatrix}\begin{pmatrix}
    f_{e}\\
    f_{o}
\end{pmatrix}.$$
\end{example}

The following theorem shows that every normal operator is orthogonally equivalent to a direct sum of operators of this form.

\begin{theorem}[\cite{Goodrich}, Theorem 3]\label{thm:spectral representation}
    Every bounded normal operator $T$ on a real Hilbert space $\mathcal{H}$ is orthogonally equivalent to an orthogonal sum $\bigoplus_{i\in I} T_{i}$ of operators on spaces $\L^{2}(\C,\mu_{i})$ of the type defined in Example \ref{ex:normal_operator}.
\end{theorem}

See also \cite[Section 4.1]{bhat-john} for a different description of the spectral representation that distinguishes between real and complex spectral values. 

\begin{proposition}\label{prop:infinite dimensional centralizer for orthogonal operator}
    Let $T$ be an orthogonal operator on a real Hilbert space $\mathcal{H}$. Denote by $Z_{T}$ its centralizer in $\O(\mathcal{H})$, $Z_{T} = \mathcal{Z}(T) \cap \O(\mathcal{H})$. Then $Z_{T}$ is an infinite-dimensional Lie group.
\end{proposition}

\begin{proof}
    The groups $\mathcal{Z}(T)$ and $\O(\mathcal{H})$ are both infinite-dimensional algebraic groups defined respectively by the polynomials $p_{1}(x) = xT-Tx$ and $p_{2}(x) = \transp{x}x - I$. So
    $$Z_{T} = \mathcal{Z}(T) \cap \O(\mathcal{H}) = \{S \in \GL(\mathcal{H}) \mid p_{1}(S) = 0, p_{2}(S) = 0\}$$
    where $\GL(\mathcal{H})$ is the set of invertible bounded linear operators of $\mathcal{H}$, and its Lie algebra is
    \begin{align*}
        \Lie(Z) &= \{S \in \mathbf{B}(\mathcal{H}) \mid \d p_{1}(I)\cdot S = 0,\ \d p_{2}(I)\cdot S = 0\}\\
        &= \{S \in \mathbf{B}(\mathcal{H}) \mid ST - TS = 0,\ \transp{S} + S = 0\}
    \end{align*}

    The dimension of $Z_{T}$ as a Lie group is that of its Lie algebra. We show that this dimension is infinite.

    Since $T \in \O(\mathcal{H})$ is normal, by Theorem \ref{thm:spectral representation}, we can decompose $\mathcal{H}$ as an orthogonal sum $\mathcal{H} = \bigoplus_{\alpha\in I} \mathcal{H}_{\alpha}$ where each $\mathcal{H}_{\alpha}$ is a space $\L^{2}(\C,\mu_{\alpha})$ for some measure $\mu_{\alpha}$ on $\C$ with compact support and which is symmetric about the real axis. The operator $T$ preserves this decomposition and can also be decomposed as an orthogonal sum $T = \bigoplus_{\alpha\in I} T_{\alpha}$. Since $T$ is orthogonal, we get $\supp(\mu_{\alpha}) \subset \S^{1}$ for every $\alpha \in I$. For $\alpha \in I$, if $f = f_{e} + f_{o} \in \mathcal{H}_{\alpha} = \L^{2}(\C,\mu_{\alpha})$ and $\lambda \in \supp(\mu_{\alpha})$, we have 
    $$T_{\alpha}f(\lambda)= \begin{pmatrix}
    M_{\Re} & -M_{\Im}\\
    M_{\Im} & M_{\Re}
\end{pmatrix}\begin{pmatrix}
    f_{e}\\
    f_{o}
\end{pmatrix}(\lambda) = \begin{pmatrix}
    \Re(\lambda)f_{e}(\lambda) - \Im(\lambda)f_{o}(\lambda)\\
    \Im(\lambda)f_{e}(\lambda) + \Re(\lambda)f_{o}(\lambda)
\end{pmatrix}.$$

Now let $S \in \mathbf{B}(\mathcal{H})$ be a bounded operator such that $S$ decomposes along $\bigoplus_{\alpha\in I} \mathcal{H}_{\alpha}$ into the orthogonal sum $\bigoplus_{\alpha\in I} S_{\alpha}$ such that each $S_{\alpha}$ acts on $\mathcal{H}_{\alpha}$ as
$$S_{\alpha}f= \begin{pmatrix}
0 & - M_{g_{\alpha}}\\
M_{g_{\alpha}} & 0
\end{pmatrix}
\begin{pmatrix}
    f_{e}\\
    f_{o}
\end{pmatrix},$$
where $g_{\alpha} : \supp(\mu_{\alpha}) \to \R$ is a continuous or measurable map.
Then $S_{\alpha}$ is anti-symmetric and commutes with $T_{\alpha}$. Thus the map $g_{\alpha} \mapsto S_{\alpha}$ provides an embedding
of $\mathcal{C}(\supp(\mu_{\alpha}),\R)$ into $\Lie(Z_{T})$. 

If there exists $\alpha \in I$ such that $\supp(\mu_{\alpha})$ is infinite, then $\mathcal{C}(\supp(\mu_{\alpha}),\R)$ is an infinite-dimensional vector space. Hence the dimension of $\Lie(Z_{T})$ is infinite. Otherwise, if all the $\supp(\mu_{\alpha})$ are finite, then $I$ is infinite. This means that $T$ has only eigenvalues and is then orthogonally equivalent to an infinite block-diagonal matrix where the blocks on the diagonal are rotation matrices $\begin{pmatrix}
    \cos(\theta) & -\sin(\theta)\\
    \sin(\theta) & \cos(\theta)
\end{pmatrix}$ corresponding to complex eigenvalues $\cos(\theta) \pm i\sin(\theta)$. In this case, any anti-symmetric matrix $\begin{pmatrix}
    0 & a\\
    -a & 0
\end{pmatrix}$ commutes with $\begin{pmatrix}
    \cos(\theta) & -\sin(\theta)\\
    \sin(\theta) & \cos(\theta)
\end{pmatrix}$ and belong to $\Lie(Z_{T})$.

\end{proof}

\begin{proposition}\label{prop:infinite dimensional centralizer for loxodromic isometry}
    Let $\gamma \in \Isom(\H^{n})$ be a loxodromic isometry and $\rho_{t} : \Isom(\H^{n}) \to \Isom(\H^{\infty})$ be an irreducible representation from Theorem \ref{thm:exotic deformation B}. Then the centralizer of $\rho_{t}(\gamma)$ in $\Isom(\H^{\infty})$ is an infinite-dimensional Lie group consisting of elliptic isometries fixing the axis of $\rho_{t}(\gamma)$ pointwise, and loxodromic isometries which share the same axis with $\rho_{t}(\gamma)$.
\end{proposition}

\begin{proof}
    The representation $\rho_{t}$ preserves the type of the isometries, so $\rho_{t}(\gamma) \in \Isom(\H^{\infty})$ is loxodromic. Denote by $\xi^{\pm} \in \partial \H^{\infty}$ the fixed points of $\rho_{t}(\gamma)$. They correspond to two independant vectors $v^{+}$ and $v^{-}$ in the light cone $\{x\in \mathcal{H}\mid Q(x) = 0\}$. We may assume that $B(v^{+},v^{-}) = -1$. Let $E = \left(\R\compactcdot v^{+}\oplus\R\compactcdot v^{-}\right)^{\perp}$ be the orthogonal complement of $\R\compactcdot v^{+}\oplus\R\compactcdot v^{-}$ in $\mathcal{H}$. The restriction of the bilinear form $B$ to $E$ is a scalar product, so $E$ is a real Hilbert space. We can write the isometry $\rho_{t}(\gamma)$ in a matrix form according to the decomposition $\mathcal{H} = \R\compactcdot v^{+} \oplus \R\compactcdot v^{-} \oplus E$,
    $$\rho_{t}(\gamma)=\left(\begin{array}{ccc}
       \lambda & 0 & 0 \\
       0 & \lambda^{-1} & 0\\
       0 & 0 & T
    \end{array}\right)$$
    where $\lambda > 1$ and $T \in \O(E)$ is an orthogonal transformation of $E$.

    If $g \in \Isom(\H^{\infty})$ commutes with $\rho_{t}(\gamma)$, then $g$ stabilizes the axis of $\rho_{t}(\gamma)$ and acts on it either as the identity or as a translation. In the first case, $g$ is elliptic and can be written in the decomposition  $\mathcal{H} = \R\compactcdot v^{+} \oplus \R\compactcdot v^{-} \oplus E$ as $$g = \left(\begin{array}{ccc}
       1 & 0 & 0 \\
       0 & 1 & 0\\
       0 & 0 & T'
    \end{array}\right)$$
    where $T' \in Z_{T}$, the centralizer of $T$ in $\O(E)$. And in the second case, $g$ has the form
    $$g = \left(\begin{array}{ccc}
       \mu & 0 & 0 \\
       0 & \mu^{-1} & 0\\
       0 & 0 & T'
    \end{array}\right)$$
    with $\mu > 0$, $\mu\neq 1$ and $T'\in Z_{T}$, $g$ is then loxodromic with translation length $\max(\mu,\mu^{-1})$.

    Conversely, any operator of the form
    $$\left(\begin{array}{ccc}
       \mu & 0 & 0 \\
       0 & \mu^{-1} & 0\\
       0 & 0 & T'
    \end{array}\right)$$
    for $\mu > 0$ and $T'\in Z_{T}$ do commute with $\rho_{t}(\gamma)$.

    Therefore, the centralizer of $\rho_{t}(\gamma)$ in $\Isom(\H^{\infty})$ can be identified with $\R^{*}_{+}\times Z_{T}$ where $Z_{T}$ is an infinite-dimensional Lie group by Proposition \ref{prop:infinite dimensional centralizer for orthogonal operator}.
\end{proof}

\subsection{Bending irreducible representations from \texorpdfstring{$\Isom(\H^{2})$}{H 2}}

The rigidity theorem of Mostow states that if $\Gamma$ is a lattice in $\PO(n,1)$ with $n \geqslant 3$, and $\rho: \Gamma \to \PO(n,1)$ is a faithful representation such that $\rho(\Gamma)$ is a lattice, then $\Gamma$ and $\rho(\Gamma)$ are conjugate in $\PO(n,1)$. In other words, there is a unique hyperbolic structure on an $n$-dimensional manifold of finite volume up to isometry. However, for some lattices $\Gamma$, there exist many faithful and discrete non-conjugate representations $\Gamma \to \PO(m,1)$, where $2 \leqslant n < m$. Bending is a way to construct examples of such non-conjugate representations. 

In this section, we only consider representations $\rho_{t}$ from $\Isom(\H^{2}) = \PO(2,1)$ to $\Isom(\H^{\infty}) = \PO(\infty,1)$ as in Theorem \ref{thm:exotic deformation B} for $n=2$. Let $0<t\leqslant 1$ and let $\Gamma < \Isom(\H^{2})$ be the fundamental group of a hyperbolic closed surface (compact, connected, orientable and without boundary), $\Gamma = \pi_{1}(S)$. We are going to use this technique to deform $\rho_{t}(\Gamma)$ in $\Isom(\H^{\infty})$. The deformation is defined depending on whether $\Gamma$ is written as a free product with amalgamation or as an HNN extension. Since $S$ is a closed surface of genus $g \geqslant 2$, we can always decompose $\Gamma$ in both forms.

Suppose first that $\alpha \subset S$ is a simple closed geodesic which separates $S$ into two connected components $S_{1}$ and $S_{2}$ that are non-contractible. Then by Van Kampen theorem, the fundamental group of $S$ is an amalgamated product, $\pi_{1}(S) = \Gamma = A*_{C} B$, where $A = \pi_{1}(S_{1})$, $B = \pi_{1}(S_{2})$ and $C = \pi_{1}(\alpha)$. In particular, the group $C$ preserves and acts cocompactly on a geodesic in $\H^{2}$ (a totally geodesic copy of $\H^{1} \simeq \R$). Moreover, this group is cyclic, let $c \in C$ be a generator. If $\mathcal{Z}(\rho_{t}(c)) = \mathcal{Z}(\rho_{t}(C))$ is the centralizer of $\rho_{t}(C)$ in $\Isom(\H^{\infty})$, then any $z \in \mathcal{Z}(\rho_{t}(c))$ provides a deformation of $\rho_{t}(\Gamma)$ by conjugating the elements of $B$:
$$\begin{array}{cccc}
    \sigma_{z,t} : & \Gamma = A *_{C} B & \to & \Isom(\H^{\infty}) \\
     & a \in A & \mapsto & \rho_{t}(a)\\
     & b \in B & \mapsto & z\rho_{t}(b)z^{-1}.
\end{array}$$

Bending can also be defined when the cocompact lattice is an HNN extension. If $\alpha$ is now a non-separating simple closed geodesic of the surface $S$, then $\pi_{1}(S) = \Gamma = A*_{C}$ where $A=\pi_{1}(S\setminus \alpha)$ and $C = \pi_{1}(\alpha)$. The surface $S\setminus \alpha$ has a single connected component but two boundary components which induce two injections of $C$ into $A$, denoted by $\iota_{1}$ and $\iota_{2}$. Let us denote by $s$ the stable letter. We have $$\Gamma = \langle A,s \mid \forall c\in C\ \iota_{1}(c) = s\iota_{2}(c)s^{-1} \rangle.$$
In a similar fashion, for $z \in \mathcal{Z}(\rho_{t}(c))$, we can deform $\rho_{t}(\Gamma)$:
$$\begin{array}{cccc}
    \sigma_{z,t} : & \Gamma = A *_{C} & \to & \Isom(\H^{\infty}) \\
     & a \in A & \mapsto & \rho_{t}(a)\\
     & s & \mapsto & z\rho_{t}(s).
\end{array}$$

\begin{lemma}\label{lem:trivial centralizer}
    Let $\Gamma < \Isom(\H^{\infty})$ be any subgroup. If $\Gamma$ has no fixed point on the boundary of $\H^{\infty}$ and leaves no proper totally geodesic subspace of $\H^{\infty}$ invariant, then the centralizer of $\Gamma$ is trivial.
\end{lemma}

\begin{proof}
Let $\Gamma < \Isom(\H^{\infty})$ and $g$ be an element in the centralizer of $\Gamma$ in $\Isom(\H^{\infty})$.
\begin{itemize}
    \item If $g$ is elliptic, then the set of fixed points $\Fix(g) \subset \H^{\infty}$ is totally geodesic and $\Gamma$-invariant.
    \item If $g$ is hyperbolic, the axis of $g$ is a geodesic in $\H^{\infty}$ which is $\Gamma$-invariant.
    \item If $g$ is parabolic, denote by $\xi \in \partial\H^{\infty}$ its fixed point in the boundary of $\H^{\infty}$. By Theorem \ref{thm:nested closed convex sets}, we have $$\{\xi\} = \bigcap_{\epsilon >0}\overline{\{y\in\H^{\infty}\mid d(y,g(y))\leqslant\epsilon\}},$$
    where all the $\{y\in\H^{\infty}\mid d(y,g(y))\leqslant\epsilon\}$ are closed, convex and $\Gamma$-invariant sets. Thus $\xi$ must be a fixed point for $\Gamma$.
\end{itemize}
It follows that when $\Gamma$ is irreducible, $g$ must be elliptic and $\Fix(g) = \H^{\infty}$, so the centralizer of $\Gamma$ is reduced to the identity element of $\Isom(\H^{\infty})$.
\end{proof}

This statement is well-known in finite dimension (see for example the second corollary of \cite[Theorem 12.2.6]{ratcliffe}).

\begin{lemma}\label{lem:same t}
    Let $\Gamma$ be the fundamental group of a closed hyperbolic surface and write $\Gamma = A*_{C}B$ or $\Gamma = A*_{C}$. Let $\mathcal{Z}(\rho_{t}(c))$ be the centralizer of $\rho_{t}(c)$, $z_{1},z_{2} \in \mathcal{Z}(\rho_{t}(c))$ and $0 < t_{1},t_{2} \leqslant 1$. If the deformations $\sigma_{z_{1},t_{1}}$ and $\sigma_{z_{2},t_{2}}$ are conjugate in $\Isom(\H^{\infty})$, then $t_{1} = t_{2}$.
\end{lemma}

\begin{proof}
    Suppose that there exists $g \in \Isom(\H^{\infty})$ such that $\sigma_{z_{1},t_{1}} = g\sigma_{z_{2},t_{2}}g^{-1}$ on $\Gamma$. For any loxodromic element $a \in A$, computing the translation lengths, we have on the one side
    $$\ell_{\H^{\infty}}(\sigma_{z_{1},t_{1}}(a)) = \ell_{\H^{\infty}}(\rho_{t_{1}}(a)) = t_{1}\ell_{\H^{2}}(a),$$
    and on the other side
    $$\ell_{\H^{\infty}}(g\sigma_{z_{2},t_{2}}(a)g^{-1})=\ell_{\H^{\infty}}(\rho_{t_{2}}(a))=t_{2}\ell_{\H^{2}}(a).$$
    Therefore, $t_{1}=t_{2}$ since $\ell_{\H^{2}}(a) > 0$.
\end{proof}

In the following, we only consider the case where $\Gamma = A*_{C} B$, \ie the geodesic $\alpha$ separates the surface $S$ into two non-contractible components so that $\Gamma = \pi_{1}(S) < \Isom(\H^{2})$ is decomposed as $\Gamma = A *_{C} B$, where $C = \langle c \rangle \simeq \Z$. Let $\mathcal{Z}(\rho_{t}(c))$ denote the centralizer of $\rho_{t}(c)$ in $\Isom(\H^{\infty})$.

\begin{lemma}\label{lem:conjugation by elliptic}
    Let $0 < t \leqslant 1$ and $A<\Isom(\H^{2})$ be a non-elementary discrete subgroup. If for all $a\in A$, $\rho_{t}(a) = g\rho_{t}(a)g^{-1}$ for some $g \in \Isom(\H^{\infty})$, then $g$ is elliptic and fixes pointwise the axis of every loxodromic isometry $\rho_{t}(a)$ for $a \in A$.
\end{lemma}

\begin{proof}
    If the isometry $g$ is parabolic, it has a unique fixed point in the boundary of $\H^{\infty}$. Since $\rho_{t}(a)$ commutes with $g$ for every $a \in A$, then $\rho_{t}(A)$ has a global fixed point in $\partial \H^{\infty}$, which is in contradiction with Proposition \ref{prop:no fixed point in boundary}.
    
    If $g$ is now loxodromic, it has two fixed points $\xi, \eta \in \partial\H^{\infty}$ in the boundary of $\H^{\infty}$. Then, the fixed point sets of all the $\rho_{t}(a)$, for $a \in A$, are contained in $\{\xi, \eta\}$, which is in contradiction with Corollary \ref{cor:image of non-elementary group}.

    Therefore $g$ must be elliptic. Let $a \in A$ be a loxodromic isometry, its image $\rho_{t}(a)$ has the same type. Denote by $L_{a}$ the axis of $\rho_{t}(a)$. Since $\rho_{t}(a) = g\rho_{t}(a)g^{-1}$, $g$ preserves $L_{a}$. Restricted to this axis, $g$ then acts either as the identity, or with exactly one fixed point, \ie as a reflection. The second case is excluded since $g$ commutes with $\rho_{t}(a)$ which acts by translation on $L_{a}$.
\end{proof}

\begin{proposition}\label{thm2.19}
    Let $0 < t \leqslant 1$. If $z_{1}, z_{2} \in \mathcal{Z}(\rho_{t}(c))$ are such that $\ell_{\H^{\infty}}(z_{1}) \neq \ell_{\H^{\infty}}(z_{2})$, then $\sigma_{z_{1},t}$ and $\sigma_{z_{2},t}$ are not conjugate in $\Isom(\H^{\infty})$.
\end{proposition}

\begin{proof}
    If $g \in \Isom(\H^{\infty})$ is such that $\sigma_{z_{1},t} = g\sigma_{z_{2},t}g^{-1}$ on $\Gamma = A*_{C} B$, then for every element $a \in A$, we have
    $\rho_{t}(a) = g \rho_{t}(a) g^{-1}$, and for every $b \in B$, we get $z_{1}\rho_{t}(b)z_{1}^{-1} = g z_{2}\rho_{t}(b)z_{2}^{-1} g^{-1}$, so $\rho_{t}(b) = (z_{1}^{-1}gz_{2}) \rho_{t}(b) (z_{1}^{-1}gz_{2})^{-1}$. Since $A$ and $B$ are both discrete and non-elementary, $g$ and $z_{1}^{-1}gz_{2}$ must be elliptic by Lemma \ref{lem:conjugation by elliptic}. If $L_{c}$ is the axis of $\rho_{t}(c)$, we have moreover that $g$ acts as the identity on $L_{c}$. For $i\in\{1,2\}$, $z_{i}$ acts on $L_{c}$ as a translation by $\ell_{\H^{\infty}}(z_{i})$ according to Proposition \ref{prop:infinite dimensional centralizer for loxodromic isometry} (if one of the $z_{i}$ is elliptic, then $\ell_{\H^{\infty}}(z_{i}) =0$ and its action of $L_{c}$ is trivial). It follows that $z_{1}^{-1}gz_{2}$ acts on $L_{c}$ as a translation by $\ell_{\H^{\infty}}(z_{2})-\ell_{\H^{\infty}}(z_{1}) \neq 0$, so $z_{1}^{-1}gz_{2}$ is a loxodromic isometry.
\end{proof}

\begin{corollary}\label{prop2.20}
    Let $0<t\leqslant 1$ and $z \in \mathcal{Z}(\rho_{t}(c))$ be a loxodromic element. The representation $\sigma_{z,t}:\Gamma \to \Isom(\H^{\infty})$ obtained by bending with $z$ is not conjugate in $\Isom(\H^{\infty})$ to the restriction of any irreducible representation $\rho_{t'}:\Isom(\H^{2})\to\Isom(\H^{\infty})$ for $0 < t' \leqslant 1$. 
\end{corollary}

\begin{proof}
    Suppose by contradiction that there exists $g \in \Isom(\H^{\infty})$ such that $\sigma_{z,t} = g\rho_{t'}g^{-1}$ on $\Gamma = A*_{C} B$. Since $\rho_{t'}|_{\Gamma} = \sigma_{\id,t'}$, we have $t=t'$ by Lemma \ref{lem:same t}. We are then left with $\sigma_{z,t} = g\sigma_{\id,t}g^{-1}$. This is impossible due to Proposition \ref{thm2.19}. 
\end{proof}

Recall from Proposition \ref{prop:infinite dimensional centralizer for loxodromic isometry} that the centralizer $\mathcal{Z}(\rho_{t}(c))$ contains infinitely many loxodromic elements. The previous results show that there is at least a one-parameter family of elements in $\mathcal{Z}(\rho_{t}(c))$ providing non-conjugate deformations of the group $\Gamma$, none of which being conjugate to the restriction of any exotic representation.

\begin{corollary}
Let $0<t\leqslant 1$ and $\Gamma$ be the fundamental group of a closed hyperbolic surface. Then the space of deformations of $\rho_{t}|_{\Gamma}$, up to conjugation in $\Isom(\H^{\infty})$, contains a one-parameter family of representations which are not pairwise conjugate, nor conjugate to the restriction of any $\rho_{t'}$, for $0<t'\leqslant 1$.
\end{corollary}

\begin{proof}
Let $\mathcal{Z}(\rho_{t}(c))$ be the centralizer of $\rho_{t}(c)$ in $\Isom(\H^{\infty})$. Write $\rho_{t}(c)$ in a suitable basis as in Proposition \ref{prop:infinite dimensional centralizer for loxodromic isometry} in the form $\rho_{t}(c)=\left(\begin{array}{ccc}
       \lambda & 0 & 0 \\
       0 & \lambda^{-1} & 0\\
       0 & 0 & T
    \end{array}\right)$, with $\lambda > 1$. Fix any $T'$ commuting with $T$. Then for every $\mu > 1$, the isometry
$$z_{\mu} = \left(\begin{array}{ccc}
       \mu & 0 & 0 \\
       0 & \mu^{-1} & 0\\
       0 & 0 & T'
\end{array}\right)$$
belongs to $\mathcal{Z}(\rho_{t}(c))$. Moreover, these isometries have distinct translation lengths $\mu$, so they satisfy Proposition \ref{thm2.19} and Corollary \ref{prop2.20}.  
\end{proof}

\begin{remark}
    In particular, the cocompact lattice $\Gamma < \Isom(\H^{2})$ has more representations into $\Isom(\H^{\infty})$ than the whole group $\Isom(\H^{2})$, and its space of convex-cocompact representations is not reduced to the restrictions of the exotic representations $\rho_{t}$.
\end{remark}

When $n\geqslant 3$, let $\rho : \Isom(\H^{n}) \to \Isom(\H^{\infty})$ be an irreducible representation. The previous construction does not always provide deformations of this representation $\rho$ in $\Isom(\H^{\infty})$. Indeed, let $M$ be a compact hyperbolic $n$-manifold with fundamental group $\Gamma < \Isom(\H^{n})$. If $N \subset M$ is a totally geodesic hypersurface, then $C := \pi_{1}(N)$ preserves and acts cocompactly on a totally geodesic copy of $\H^{n-1}$ inside $\H^{n}$ and $N \simeq \H^{n-1}/C$. As previously, if $N$ is separates $M$ into two connected components $M_{1}$ and $M_{2}$ which are non-contractible, we can split $\Gamma$ into an amalgamated product $\Gamma = A*_{C} B$ with $A = \pi_{1}(M_{1})$ and $\pi_{1}(M_{2})$. And if $N$ is non-separating, then $\Gamma$ writes as the HNN extension $\Gamma = A*_{C}$ where $A = \pi_{1}(M\setminus N)$. The universal covers $\widetilde{M_{1}}$ and $\widetilde{M_{2}}$ are simply connected complete hyperbolic manifolds with totally geodesic boundaries, thus they are isometric to some intersection of half-spaces in $\H^{n}$ with disjoint boundaries (\cite[Theorem 3.5.2]{martelli}). Their fundamental groups $A$ and $B$ are thus convex-cocompact in $\Isom(\H^{n})$. 

The group $C$ is finitely generated as the fundamental group of a compact hyperbolic manifold, however it is not cyclic anymore. Let $S =\{c_{1},\dots,c_{m}\}$ be a finite generating set for $C$. The centralizer of $\rho(C)$ in $\Isom(\H^{\infty})$ is the intersection of the $\mathcal{Z}(\rho(c_{k}))$, $1 \leqslant k \leqslant m$, and might be trivial even though each $\mathcal{Z}(\rho(c_{k}))$ is as described in Proposition \ref{prop:infinite dimensional centralizer for loxodromic isometry}.

\begin{remark}
    Note that hyperbolic manifolds with totally geodesic hypersurfaces do exist in any dimension, see for example \cite[chapter 7]{JM} and \cite[chapter 3]{kapovich07} for constructions using arithmetic lattices. It was actually proved in \cite{bader-fisher-miller-stover} that a finite volume hyperbolic manifold with infinitely many totally geodesic hypersurfaces is necessarily arithmetic.
    
\end{remark}

\bibliographystyle{alpha}
\bibliography{main.bib}

\noindent{\sc School of Mathematics, Korea Institute for Advanced Study (KIAS), 02455 Seoul, Korea}

\noindent{\it Email address:} {\tt \href{mailto:davidxu@kias.re.kr}{davidxu@kias.re.kr}}

\end{document}